\documentclass[12pt]{amsart}
\usepackage{amsfonts,latexsym,graphicx,amssymb,amsmath,url,amsthm, enumerate}
\usepackage[hmargin=3cm,vmargin=2.9cm]{geometry}
\usepackage[usenames,dvipsnames]{color}
\usepackage{tikz-cd} 

\usetikzlibrary{shapes,arrows,decorations.pathmorphing}

\newtheorem{theorem}{Theorem}
\newtheorem{lemma}{Lemma}

\theoremstyle{definition}
\theoremstyle{remark}

\newcommand{\N}{\mathbb N}
\newcommand{\Q}{\mathbb Q}

\numberwithin{equation}{section}
\numberwithin{lemma}{section}


\usepackage{hyperref}
\begin{document}

\title{Rational lines on cubic hypersurfaces II}

\author{Julia Brandes}
\address{JB: Mathematical Sciences, Chalmers Institute of Technology and University of Gothenburg, 412 96 G{\"o}teborg, Sweden}
\email{brjulia@chalmers.se}

\author{Rainer Dietmann}
\address{RD: Department of Mathematics, Royal Holloway, University of London, Egham,
TW20 0EX, UK}
\email{rainer.dietmann@rhul.ac.uk}

\author{David B. Leep}
\address{DL: Department of Mathematics\\University of Kentucky\\Lexington, KY
40506-0027 U.S.A.}
\email{leep@uky.edu}

\begin{abstract}
	We show that any rational cubic hypersurface of dimension at least $33$ defined over a number field $K$ vanishes on a $K$-rational projective line, reducing the previous lower bound of Wooley by two. For $K=\Q$ we can reduce the bound to $29$. The main ingredients are a result on linear spaces on quadratic forms over suitable non-real quadratic field extensions, and recent work of Bernert and Hochfilzer on cubic forms over imaginary quadratic number fields for the rational case.
\end{abstract}

\maketitle

\section{Introduction}

A celebrated theorem of Birch \cite{Bir:57} asserts that any rationally defined projective hypersurface of odd degree contains large rational linear spaces, provided only that the dimension of the hypersurface be sufficiently large in terms of its degree and the dimension of the linear space. Unfortunately, his methods, although constructive, are famously inefficient and believed to be quite far from the truth. This is true even in some of the simplest cases, such as that of finding rational linear spaces, or even just lines, on cubic hypersurfaces. In particular, Wooley \cite[Theorem 2(b)]{W} proved that every cubic hypersurface defined by a homogeneous polynomial in at least $37$ variables with coefficients in any number field $K$ contains a $K$-rational line. For $K=\Q$ and under the additional assumption that the cubic hypersurface is smooth, the first two authors \cite[Theorem~1.1]{BD} reduced that number to $31$. The aim of this note is to improve on Wooley's result for general $K$, thereby also extending the result of \cite{BD} to the singular case. 

For any number field $L \mid \Q$ 
denote by $\gamma_{L}$ the least integer $n$ with the property that any
$L$-rational cubic form in at least $n$ variables has a non-trivial $L$-rational zero. Note that this is well-defined since $\gamma_L \le 16$ for any number field $L$ was confirmed by Pleasants \cite{P}.
For a given number field $K \mid \Q$ define then
\begin{align*}
	\delta_K = \sup_{L} \gamma_L,
\end{align*}
where the supremum is over all non-real quadratic extensions $L \mid K$. Note that $\delta_K \le 16$ in general by the aforementioned work of Pleasants, and that $\delta_\Q \le 14$ by recent work of Bernert and Hochfilzer \cite{BH} who established $\gamma_L \le 14$ for all imaginary quadratic number fields $L$. In this notation,
we obtain the following result on linear spaces on cubic hypersurfaces.

\begin{theorem}	\label{f18}
	Let $K \mid \Q$ be a number field, let $n \in \N$ and let $C \in K[X_1, \ldots, X_n]$ be a cubic form. If $n \ge 2\delta_K+3$, then $C$ vanishes on a $K$-rational projective line.
\end{theorem}

Partly as a consequence of Theorem~\ref{f18}, our main theorem is now as follows.  
\begin{theorem}	\label{top_gun}
	Let $K \mid \Q$ be a number field, and let $C(X_1, \ldots, X_n) \in K[X_1, \ldots, X_n]$ be a cubic form. Then $C$ vanishes on a $K$-rational projective line provided that
	\begin{itemize}
  		\item[(a)] $n \ge 35$ in general; or
  		\item[(b)] $n \ge 33$ and $K$ is imaginary quadratic; or
  		\item[(c)] $n \ge 31$ and $K=\Q$.
	\end{itemize}
\end{theorem}

Observe that necessarily $\delta_K \ge 10$, so $n \ge 23$ is a hard limit of the method. We remark that it is quite unclear what the `true' bound should be, but it was shown in \cite[Theorem~1.4]{BD} that for $K=\Q$ at the very least $12$ variables are needed in order to avoid potential $p$-adic obstructions. \smallskip

In our proof we follow largely the same strategy as in \cite{W}.
In other words, we begin by finding a point $\mathbf x \in K^n$
satisfying $C(\mathbf x)=0$, and then seek another point $\mathbf y$
having the property that the pair $(\mathbf x, \mathbf y)$ spans a projective line on the hypersurface $C=0$. In order to find the point $\mathbf y$, we are led to find non-trivial solutions to a system of one linear, one quadratic and one cubic equation. The key to the strategy is to find a large linear space on which the quadratic equation vanishes, and then solve the cubic equation after restricting to this linear space. Since the quadratic equation may be definite unless
$K$ is totally imaginary, it has been customary to solve this problem over
a suitable non-real quadratic field extension $L \mid K$ to ensure solubility. The novelty of our approach lies in how this choice is made. Unlike Wooley \cite{W}, who used $L=K(i)$, we optimise our choice of $L$, which enables
us to find high-dimensional $L$-linear spaces on the quadratic form more efficiently. Our main result here is the following. 

\begin{theorem} \label{bell}
	Let $K \mid \Q$ be a number field, let $k,n \in \N$, and let 
	$$
		Q(X_1, \ldots, X_n) \in K[X_1, \ldots, X_n]
	$$
	be a quadratic form. If $n \ge 2k + 1$, then there exists a non-real quadratic field extension $L$ of $K$ such that $Q$ vanishes on an affine $L$-linear space of dimension at least $k$. 
\end{theorem}

For comparison, it follows from Leep \cite[Corollary 2.4]{Leep} that $Q$ vanishes on an affine $k$-dimensional linear space over \emph{all} totally
imaginary number fields when $n \ge 2k+3$.

Note that the condition $n \ge 2 k + 1$ in Theorem~\ref{bell} is sharp in general for all number fields $K$ as will be shown in the last section.

\textbf{Acknowledgements.} During the production of this manuscript,
the first author was supported by Project Grant 2022-03717 from
Vetenskapsr{\aa}det (Swedish Science Foundation). We are also grateful
to Christian Bernert and Leonhard Hochfilzer for alerting us to their work
\cite{BH} and sharing a preprint.

\section{Notation and preliminaries}

An algebraic number field $K$ is a finite extension of the rational field $\mathbb{Q}$.
A totally imaginary
number field is a number field having no real embeddings.
If $L \mid K$ is an extension of a number field $K$ where $L$ is
totally imaginary, we call $L \mid K$ a non-real extension.
Let $\mathcal{O}_K$ denote the ring of integers of $K$.  We will write $\mathcal{O}$ if the field $K$ is understood.  

Let $\mathfrak{p}$ be a nonzero prime ideal in $\mathcal{O}_K$, then we write $v_{\mathfrak{p}}(\cdot)$ for the normalised $\mathfrak p$-adic valuation in $K$. Let $K_{\mathfrak p}$ denote the completion of $K$ with respect to $v_{\mathfrak p}$, let $\mathcal{O}_{K,\mathfrak{p}}=\mathcal{O}_{\mathfrak{p}}$ be the ring of integers of $K_{\mathfrak p}$, and let $\mathfrak k_{\mathfrak p} = \mathcal{O}_{\mathfrak{p}} / \pi \mathcal{O}_{\mathfrak{p}}$ denote the residue field, $\pi$ being a uniformising element of $\mathcal O_{\mathfrak p}$. 
We also let $K_{\mathfrak p}$ denote the completion of $K$ with respect to any archimedean valuation of $K$, that is, the completion of $K$ with respect to any real or complex embedding of $K$.

For any field $K$, we write $K^\times$ and $(K^{\times})^2$ for the
non-zero elements and the non-zero squares of $K$, respectively. If $d \in (K^{\times})^2$, by some abuse of notation we write $\sqrt d \in K$ for either of the two solutions to $x^2 = d$; the specific choice here does not enter in our ensuing arguments.
Further, for a field $K$ and $c_1, \ldots, c_m \in K$, let $\langle c_1, \ldots, c_m \rangle$ denote the quadratic form $c_1 x_1^2 + \cdots + c_m x_m^2$. Let $\mathbb{H} = \langle 1, -1 \rangle$ and let $m \mathbb{H}$ denote an orthogonal direct sum of $m$ copies of $\mathbb{H}$.  That is, $m \mathbb{H}$ is the $2m$-dimensional quadratic form $\langle 1, -1, \ldots, 1, -1 \rangle$. We let $\perp$ denote an orthogonal direct sum.  Thus 
\[
	\langle c_1, \ldots, c_m \rangle \perp \langle d_1, \ldots, d_n \rangle = \langle c_1, \ldots, c_m, d_1, \ldots, d_n \rangle.
\]

One of our main ingredients is a local-to-global principle for linear spaces of zeros of quadratic forms, which extends the Hasse--Minkowski Theorem.

\begin{lemma}\label{hauptlemma}
	Let $k$ and $n$ be positive rational integers, let $K$ be a number field, and let $Q(X_1, \ldots, X_n) \in K[X_1, \ldots, X_n]$ be a non-singular quadratic form. Then $Q$ vanishes on a $K$-linear space of dimension $k$ if and only if for all completions $K_{\mathfrak p}$ of $K$ the form $Q$ vanishes on a $K_{\mathfrak p}$-linear space of dimension $k$.
\end{lemma}
\begin{proof}
	This is a special case of Theorem 66:3 in \cite{OM}. In the notation of the cited theorem, we need to put $U = \mathbb{H} \perp \ldots \perp \mathbb{H}$ (a $k$-fold sum of hyperbolic planes $\mathbb{H}$) and $V = K^n$.
\end{proof}

The following result provides local input for our application of the extended Hasse--Minkowski Theorem in the easier locally non-singular
case.

\begin{lemma}\label{lemma 2}
	Let $K$ be a number field, let $k \in \N$ and $n=2k+1$, let $a_1, \ldots, a_n \in \mathcal{O}_K$, and let $\mathfrak{p}$ be a nonzero prime ideal in $\mathcal{O}_K$ with $2a_1 a_2 \cdots a_n \notin \mathfrak{p}$.  Then 
	\[
		\langle a_1, \ldots, a_n \rangle \cong_{K_{\mathfrak{p}}} k \mathbb{H} \perp \langle (-1)^k a_1 a_2 \cdots a_n \rangle
	\]
	and $Q$ vanishes on a $k$-dimensional linear space over $K_{\mathfrak{p}}$.
\end{lemma}

\begin{proof}
	The proof is by induction on $k \ge 1$.
 Chevalley's theorem applied to the finite field $\mathfrak k_{\mathfrak p}$ implies that $\langle a_1, a_2, a_3 \rangle$ has a non-trivial zero
 over $\mathfrak k_{\mathfrak p}$.
 As $2 a_1 a_2 a_3 \notin \mathfrak{p}$, any such nontrivial zero is a
nonsingular zero and thus lifts by Hensel's lemma to a nontrivial zero over
$\mathcal{O}_{\mathfrak{p}}$. Thus, $\langle a_1, a_2, a_3 \rangle$ splits off a hyperbolic plane over $K_{\mathfrak{p}}$ so that 
$$
	\langle a_1, a_2, a_3 \rangle \cong_{K_{\mathfrak{p}}} \mathbb{H} \perp \langle c \rangle
$$ for some $c \in K_{\mathfrak{p}}^{\times}$ (see \cite[42.10, p.~95]{OM}).

	Comparing determinants and using that $\det \mathbb{H} = -1$ shows that we can make the choice $c = -a_1 a_2 a_3$.  Thus
	\[
		\langle a_1, a_2, a_3 \rangle \cong_{K_{\mathfrak{p}}} \mathbb{H} \perp \langle -a_1 a_2 a_3 \rangle.
	\]
	This gives 
	\[ 
		\langle a_1, \ldots, a_n \rangle \cong_{K_{\mathfrak{p}}} \mathbb{H} \perp \langle  -a_1a_2a_3, a_4, \ldots, a_n\rangle.
	\]
	Since $2(-a_1 a_2 a_3) a_4 a_5 \cdots a_n \notin \mathfrak{p}$, it follows by induction that 
	\begin{align*}
		\langle a_1, \ldots, a_n \rangle &\cong_{K_{\mathfrak{p}}} \mathbb{H} \perp (k-1) \mathbb{H} \perp \langle (-1)^{k-1} (-a_1 a_2a_3) a_4 \cdots a_n \rangle\\
		& \cong_{K_{\mathfrak{p}}} k \mathbb{H} \perp \langle (-1)^k a_1 a_2 \cdots a_n \rangle.
	\end{align*}
	As $\mathbb{H}$ vanishes on a $1$-dimensional space, it follows that a direct sum of $k$ of these forms vanishes on a $k$-dimensional space over $K_{\mathfrak{p}}$.
\end{proof}

\section{Quadratic forms under quadratic field extensions}

We are now in a position to prove Theorem \ref{bell}.
The quadratic form $Q$ can be diagonalised over $K$.  Thus there exist $a_1, \ldots, a_n \in K$ such that
\[
	Q \cong_{K} \langle a_1, \ldots, a_n \rangle,
\]
where $a_1, \ldots, a_n \in K$. 

Without loss of generality we can assume that $n=2k+1$, and we will further assume first that $a_1 \cdots a_n \ne 0$. At the end, we will handle the case when $a_1 \cdots a_n = 0$.

Note that after multiplying each $a_i$ by an appropriate nonzero square in $\mathcal{O}$ if necessary, we may assume that each $a_i \in \mathcal{O}$. From Lemma \ref{lemma 2} we already know that $Q$ vanishes on a $K_{\mathfrak{p}}$-linear space of affine dimension $k$ for all prime ideals $\mathfrak{p} \subset \mathcal{O}$ with $2a_1 a_2 \cdots a_n \notin \mathfrak{p}$.

Now let $\mathcal{P} = \{\mathfrak{p}_1, \ldots, \mathfrak{p}_r\}$ be the finite set of nonzero prime ideals $\mathfrak{p}$ in $\mathcal{O}$ such that $2a_1 a_2 \cdots a_n \in \mathfrak{p}$. By weak approximation (see \cite{OM}[11:8, p.~8]), choose $d \in \mathcal{O}$ such that the $\mathfrak{p}$-adic valuation gives $v_{\mathfrak{p}}(d) = 1$ for each $\mathfrak{p} \in \mathcal{P}$ and $d < 0$ with respect to each of the finitely many real embeddings of $K$.

Let $L = K(\sqrt{d})$. Then $L$ is a non-real quadratic extension of $K$. Indeed, we have 
\[
	[K_{\mathfrak{p}}(\sqrt{d}): K_{\mathfrak{p}}] = 2 \text{ for each $\mathfrak{p} \in \mathcal{P}$},
\]
because $v_{\mathfrak{p}}(d) = 1$:
Suppose that $\sqrt{d} \in K_{\mathfrak{p}}$ for some 
$\mathfrak{p} \in \mathcal{P}$. Then 
\[
	2v_{\mathfrak{p}}(\sqrt{d}) = v_{\mathfrak{p}}\big( \sqrt{d}^{\,\,2} \big) = v_{\mathfrak{p}}(d) = 1.
\]   
Thus $v_{\mathfrak{p}}(\sqrt{d}) = 1/2$, a contradiction as we had assumed $v_{\mathfrak p}$ to be normalised.

Recall that every regular quadratic form over $K_{\mathfrak{p}}$ in five or more variables is isotropic and thus splits off one hyperbolic plane as an orthogonal direct summand (see \cite[Chapter~VI, Theorem 2.12, p.~158]{Lam} or \cite{OM}[63:19, p.~170]).
Using this fact and recalling that $n = 2(k-1) + 3$, we see for each $\mathfrak{p} \in \mathcal{P}$ that
\[
	\langle a_1, \ldots, a_n \rangle \cong_{K_{\mathfrak{p}}} (k-1) \mathbb{H} \perp \langle  r,s,t \rangle,
\]
where $r,s,t \in K_{\mathfrak{p}}^{\times}$ depend on $\mathfrak{p}$.

We now show that $\langle  r,s,t \rangle$ is isotropic over $K_{\mathfrak{p}}(\sqrt{d})$.
The quadratic form $\langle r, s, t, rstd \rangle$ is isotropic over $K_{\mathfrak{p}}$ because its determinant lies in $d K_{\mathfrak{p}}^2$ and $d$ is not a square in $K_{\mathfrak{p}}$ (\cite[Chapter VI, Corollary 2.15, p.~159-160]{Lam}).
This implies that $\langle r, s, t, rst \rangle$ is isotropic over $K_{\mathfrak{p}}(\sqrt{d})$ because $d$ is a square in $K_{\mathfrak{p}}(\sqrt{d})$.
It follows that $\langle r, s, t \rangle$ is isotropic over $K_{\mathfrak{p}}(\sqrt{d})$ by \cite[42:12, p.~95]{OM}.
Thus $\langle r, s, t \rangle$ splits off a hyperbolic plane over $K_{\mathfrak{p}}(\sqrt{d})$ by \cite[42:10, p.~95]{OM}, so $\langle r, s, t \rangle \cong_{K_{\mathfrak{p}}(\sqrt{d})} \mathbb{H} \perp \langle c \rangle$ for some $c \in  K_{\mathfrak{p}}(\sqrt{d})^{\times}$. Consequently,
for each $\mathfrak{p} \in \mathcal{P}$ we conclude that
\begin{align*}
	\langle a_1, \ldots, a_n \rangle& \cong_{K_{\mathfrak{p}}} (k-1) \mathbb{H} \perp \langle  r,s,t \rangle \\
	& \cong_{K_{\mathfrak{p}}(\sqrt{d})} k \mathbb{H} \perp \langle c \rangle \\
	& \cong_{K_{\mathfrak{p}}(\sqrt{d})} k \mathbb{H} \perp \langle (-1)^k a_1 a_2 \cdots a_n \rangle,
\end{align*}
because the determinant on the right hand side must be $a_1 a_2 \cdots a_n$ times a square. It follows that $\langle a_1, \ldots, a_n \rangle$ and therefore $Q$ vanishes on a $K_{\mathfrak{p}}(\sqrt{d})$-linear space of affine dimension $k$ for every nonzero prime ideal $\mathfrak{p} \in \mathcal{O}$. The same holds trivially over each complex completion of $K$. There are no real completions because $L$ is a non-real quadratic extension of $K$.

Since every nonarchimedean completion of $L = K(\sqrt{d})$ has the form $K_{\mathfrak{p}}(\sqrt{d})$, the extended version of the Hasse--Minkowski Theorem given in Lemma \ref{hauptlemma} implies that $\langle a_1, \ldots, a_n \rangle$ and thus $Q$ vanishes on a $k$-dimensional space linear over $L$. \bigskip

We now deal with the general case where we might have $a_1 \cdots a_n = 0$. Suppose that $a_1 \cdots a_j \ne 0$ and $a_{j+1} = \cdots = a_n = 0$. Then $\langle a_1, \ldots, a_n \rangle$ vanishes trivially on an $(n-j)$-dimensional space over $K$. If $n-j \ge k$, this establishes the desired conclusion, so we may suppose that $n-j<k$. Since $n = 2k+1$ and $n \ge j$, we have $2n \ge 2k+1 + j$, which implies that
$j \ge 2(k - (n-j)) + 1$. We now apply the above work to the subform $\langle a_1, \ldots, a_m \rangle$
where $m = 2(k - (n-j)) + 1 \le j.$ Then there exists a non-real quadratic extension $L$ of $K$ such that $\langle a_1, \ldots, a_m \rangle$ vanishes on an $L$-linear space of dimension at least $k - (n-j)$, and thus the same holds for $\langle a_1, \ldots, a_j \rangle$. Then $\langle a_1, \ldots, a_j, a_{j+1}, \ldots, a_n \rangle$ vanishes on an $L$-linear space of dimension at least $k - (n-j) + (n-j) = k$.

\section{Proof of the main result}

To prove Theorem~\ref{f18}, recall that $C \in K[X_1, \ldots, X_n]$ is a
cubic form with $n \ge 2\delta_K+3$.
As usual, we introduce a $K$-rational symmetric trilinear form $\Phi$ having the property that
$C(\mathbf{X})=\Phi(\mathbf{X}, \mathbf{X}, \mathbf{X})$.
As $\delta_K \ge 10$ we have $n \ge 2\delta_K+3 \ge 16$,
so by \cite{P} there exists a base point $\mathbf{x} \in K^n \backslash \{\mathbf{0}\}$ with $C(\mathbf{x})=0$.
Our task is now to find a point $\mathbf y \in K^n$
such that the pair $(\mathbf x, \mathbf y)$ spans a line that is entirely contained in the hypersurface $C=0$. 

Put
\[
	\Lambda(\mathbf{Y}) = \Phi(\mathbf{x}, \mathbf{x}, \mathbf{Y}) \quad \text{ and } \quad Q(\mathbf{Y}) = \Phi(\mathbf{x}, \mathbf{Y}, \mathbf{Y}),
\]
noting that 
\[
	C(\mathbf x + \mu \mathbf Y) = 3\mu \Lambda(\mathbf Y) + 3\mu^2 Q(\mathbf Y) + \mu^3 C(\mathbf Y).
\]
To ensure that the points $\mathbf x$ and $\mathbf y$ are linearly independent, choose any linear form $M \in K[X_1, \ldots, X_n]$ with $M(\mathbf{x}) \ne 0$.
We thus want to solve the system of two linear, one quadratic and one cubic 
$K$-rational Diophantine equations
\begin{equation*}
  	\Lambda(\mathbf{y})=M(\mathbf{y})=Q(\mathbf{y})=C(\mathbf{y})=0.
\end{equation*}
Denote by $H \simeq K^{m}$ the affine space defined by $M=\Lambda=0$, and denote by $C_1$ and $Q_1$ the restrictions of the forms $C$ and $Q$ to $H$.
As $m \ge n-2 \ge 2\delta_K + 1$, Theorem~\ref{bell} shows that there exists a non-real quadratic field extension $L$ of $K$ such that $Q_1$ vanishes on an $L$-linear space $V$ of dimension at least $\delta_K$. Writing $C_2$ for the restriction of $C_1$ to $V$, it now follows from the definition of $\delta_K$ that $C_2$ has a non-trivial zero $\mathbf y_1 \in V$. This point can in a natural manner be lifted to a non-zero point $\mathbf y \in L^n$. We thus have found a pair of linearly independent points
$\mathbf x \in K^n$ and $\mathbf y \in L^n$,
such that $C$ vanishes on the $L$-rational linear space spanned by
$\mathbf x$ and $\mathbf y$.
In this situation, Lemma 2.2 in \cite{W} enables us to pull back the
$L$-linear space spanned by $\mathbf{x}$ and $\mathbf{y}$ to a
$K$-linear two-dimensional space. We thus obtain a $K$-rational projective
line on $C=0$ as required.

Regarding the proof of Theorem~\ref{top_gun}, parts (a) and (c) follow immediately from Theorem~\ref{f18} after inserting the bounds $\delta_K \le 16$ and $\delta_\Q \le 14$, respectively. For part (b), we first follow the previous proof of Theorem \ref{f18}, but then instead of invoking Theorem \ref{bell} and working over a quadratic extension, stay in the ground field $K$. This is possible as $K$ is imaginary quadratic, so we can use \cite[Corollary 2.4]{Leep} instead: As $n-2 \ge 2 \gamma_K+3$ due to $\gamma_K \le 14$, we can find a linear space $V$ of dimension at least $\gamma_K$ on which $Q_1$ vanishes. We can then continue in a similar but simpler way as before as we
do not need to pull back our line from a quadratic extension.

\section{Sharpness}

We now show that the condition $n \ge 2k + 1$ in Theorem \ref{bell} is sharp in general. Suppose that $n = 2k$ with $k \ge 2$. We will find $a,b,d \in K^{\times}$ and a nonzero prime ideal $\mathfrak{p}$ in $\mathcal{O}$ that satisfy the following conditions: \begin{enumerate}
	\item $d \notin (K^{\times})^2$.
	\item $d \notin \mathfrak{p}$ and $d \in (K_{\mathfrak{p}}^{\times})^2$.
	\item $\langle 1, a, b, abd \rangle$ is anisotropic over $K_{\mathfrak{p}}$.
\end{enumerate}
Suppose first that $a,b,d$ and $\mathfrak{p}$ have been found as above. We let
\[
	Q = X_1^2 + a X_2^2 + b X_3^2 + abd X_4^2 + X_5 X_6 + \cdots + X_{n-1} X_n.
\] 
Since $Q$ already splits off $\frac{n-4}{2}$ hyperbolic planes, in order for $Q$ to vanish on an affine $L$-linear space of dimension $k = \frac{n}{2}$, for some quadratic field extension $L$ of $K$, the quadratic form 
\[
	G = X_1^2 + a X_2^2 + b X_3^2 + abd X_4^2
\] 
would have to vanish on a two-dimensional $L$-linear space. This is possible only if $\det G$ is a square in $L$, which forces $L = K(\sqrt{d})$.  We have $[L:K] = 2$ by condition (1).

However, $G$ is anisotropic over $K_{\mathfrak{p}} = K_{\mathfrak{p}}(\sqrt{d})$ by conditions (2) and (3), and thus $G$ is anisotropic over $L$ because $L$ is isomorphic to a subfield of $K_{\mathfrak{p}}(\sqrt{d})$. We conclude that there is no quadratic extension $L$ of $K$ such that $Q$ vanishes on an $L$-linear space of dimension $k$.

Now we find $a,b,d$ and $\mathfrak{p}$ satisfying conditions (1)--(3). Let $\mathfrak{p}$ and $\mathfrak{q}$ be distinct nonzero prime ideals in $\mathcal{O}$. By weak approximation, choose $d \in \mathcal{O}$ such that $v_{\mathfrak{q}}(d) = 1$ and $d$ is sufficiently close to $1$ in the $\mathfrak{p}$-adic topology such that $d \in (K_{\mathfrak{p}}^{\times})^2$. Then (1) holds because $v_{\mathfrak{q}}(d) = 1$ and (2) holds by the $\mathfrak{p}$-adic condition on $d$.
As $K$ is dense in $K_{\mathfrak p}$, by \cite{Lam}[Chapter VI, Corollary 2.11, p.~158] or \cite{OM}[63:18, p.~170],
there exist $a,b \in K^{\times}$ such that $\langle 1, a, b, ab \rangle$ is anisotropic over $K_{\mathfrak{p}}$. Since $d \in (K_{\mathfrak{p}}^{\times})^2$, it follows that (3) holds.

\bibliography{bdl}
\bibliographystyle{plain}

\end{document}